\documentclass{amsart}

 \usepackage{amsmath} 

\usepackage{amsthm}
\usepackage{amsfonts}
\usepackage{amssymb}
\usepackage{hyperref}
\usepackage{stmaryrd}
\usepackage{mathtools}
\usepackage{mathrsfs}
\usepackage[capitalize]{cleveref}

\newcommand{\dc}{\#^\mathrm{dc}}

\newcommand{\e}{\varepsilon}
\newcommand{\tp}{\mathrm{tp}}

\newcommand{\Lc}{\mathcal{L}}
\newcommand{\Uc}{\mathcal{U}}
\newcommand{\Fc}{\mathcal{F}}

\DeclareMathOperator{\cl}{cl}

\providecommand{\dotdiv}{
  \mathbin{
    \vphantom{+}
    \text{
      \mathsurround=0pt 
      \ooalign{
        \noalign{\kern-.35ex}
        \hidewidth$\smash{\cdot}$\hidewidth\cr 
        \noalign{\kern.35ex}
        $-$\cr 
      }%
    }%
  }%
}

\newtheorem{thm}{Theorem}[section]
\newtheorem{prop}[thm]{Proposition}
\newtheorem{lem}[thm]{Lemma}
\newtheorem{cor}[thm]{Corollary}

\newtheorem{quest}[thm]{Question}

\theoremstyle{definition}
\newtheorem{defn}[thm]{Definition}
\newtheorem{nota}[thm]{Notation}

\newcommand{\Ff}{\mathfrak{F}}
\newcommand{\Pf}{\mathfrak{P}}

\newcommand{\Rb}{\mathbb{R}}

\newcommand{\res}{{\upharpoonright}}

\DeclareMathOperator{\diam}{diam}

\newcommand{\Kbf}{\mathbf{K}}

\DeclareMathOperator{\dom}{dom}

\DeclareMathOperator{\Th}{Th}

\newcommand{\der}{\partial}

\DeclareMathOperator{\ccl}{ccl}

\begin{document}

\title{Topometric Characterization of Type Spaces in Continuous Logic}
\author{James Hanson}
\email{jehanson2@wisc.edu}
\address{Department of Mathematics, University of Wisconsin--Madison, 480 Lincoln Dr., Madison, WI 53706}
\date{\today}


\begin{abstract}
  We show that a topometric space $X$ is topometrically isomorphic to a type space of some continuous first-order theory if and only if $X$ is compact and has an open metric (i.e., satisfies that $\{p : d(p,U) < \e\}$ is open for every open $U$ and $\e > 0$). Furthermore, we show that this can always be accomplished with a stable theory.
\end{abstract}

\maketitle
\vspace{-2em}

\section*{Introduction}

\noindent Continuous first-order logic, introduced in \cite{MTFMS}, is a generalization of discrete first-order logic suited for studying structures with natural metrics, such as C$^\ast$-algebras, valued fields, and $\Rb$-trees. Such structures are referred to as metric structures. In continuous logic, formulas take on arbitrary real values, arbitrary continuous functions are used as connectives, and supremum and infimum take on the role of quantifiers. 

A notably subtle aspect of the generalization to continuous logic is the treatment of definable sets. A subset $D$ of a metric structure $M$ is \emph{definable} if $d(x,D)\coloneqq \inf_{y \in D}d(x,y)$ is a definable predicate. 

In discrete logic, definable sets have a purely topological characterization in terms of clopen subsets of type space. In continuous logic, there is important extra structure on type spaces, namely the induced metric given by
\[
  d(p,q) \coloneqq \inf \{ d(a,b) : a \models p\text{, }b \models q\}.
\]
This metric induces a topology on type space that is generally strictly finer
than the compact logic topology. The induced metric enjoys certain strong
compatibility properties with the logic topology, which Ben Yaacov abstracted
to a general theory of topometric spaces in \cite{BenYaacov2008}.
\begin{defn}
  A \emph{topometric space} $(X,\tau,\der)$ is a set $X$ together with a topology $\tau$ and a metric $\der$ such that the metric refines the topology and is lower semi-continuous (i.e., has $\{(x,y)\in X^2 : \der(x,y) \leq \e\}$ closed for every $\e > 0$).
\end{defn}
Just as with the purely topological characterization of definable sets in discrete logic, there is a purely topometric characterization of definable sets in continuous logic. A set $D \subseteq S_n(T)$ corresponds to a definable set if and only if it is closed and has $D^{<\e}\coloneqq \{x \in S_n(T) : d(x,D) < \e\}$ open for every $\e > 0$.\footnote{By an abuse of terminology, we also refer to such sets of types as \emph{definable}.} This means that the family of definable sets in a given type space can be read off from its topometric structure alone.

In \cite{CatCon}, the author showed that in $\omega$-stable (and more generally totally transcendental) theories, definable sets are prevalent in a precise sense (referred to there as \emph{dictionaricness}) that permits, at least approximately, many of the manipulations of definable sets that are trivial in discrete logic. In general, though, it is known that there is a weakly minimal theory $T$ such that $S_1(T)$ has cardinality $2^{\aleph_0}$ but no non-trivial definable sets. (See Example 4.2 in \cite{CatCon}.) 

These facts motivate a desire to solve the following question: Which topometric spaces are type spaces of some continuous theory, and what restrictions, if any, do various model-theoretic tameness conditions impose on the topometry types of type spaces? Despite the existence of superstable theories that are poorly behaved with regards to definable sets, there might be, in principle, some subtler regularity imposed on them by stability or superstability.

In discrete logic, this question is very easy to answer. Given any totally disconnected compact Hausdorff space $X$, there is a weakly minimal theory $T$ such that $S_1(T) \cong X$. Furthermore, if $X$ is scattered (i.e., has ordinal Cantor-Bendixson rank), then $T$ can be taken to be totally transcendental. 

The metric $d$ on a given type space $S_n(T)$ has an additional regularity property not shared by compact topometric spaces in general, identified in \cite{OnPert} and referred to provisionally there as openness. For lack of a better alternative, we will adopt the same terminology.
\begin{defn}
  A topometric space $(X,\tau,\der)$ has an \emph{open} metric if for any open set $U$ and any $\e > 0$, $U^{<\e}$ is an open set.
\end{defn}
In this paper we will show that compactness and openness of metric precisely characterize the topometry types of type spaces in continuous logic. Furthermore, we will show that any compact topometric space with an open metric is topometrically isomorphic to a type space of a stable theory.

Although our result does resolve the general question completely, it still leaves open the characterization of the topometry types of type spaces in totally transcendental and superstable theories.

\begin{nota}
  In any topological or topometric space $X$, we will write $\cl A$ for the topological closure of the set $A$.
\end{nota}

Note that for the purposes of this paper, closures in topometric spaces are always taken with regards to the topology rather than the metric.

Finally, we should draw the reader's attention to the work of Carlisle and Henson on the model theory of $\Rb$-trees \cite{Carlisle2020}. Although we do not use very many results from \cite{Carlisle2020} directly, many of the technical ideas used here are indebted to that paper.

\section{Generic $\Rb$-Forest-Embeddable Structures}


\begin{defn}
  Given a compact topometric space $(X,\tau,\der)$ with finite positive diameter, we define an associated language $\Lc_X$ consisting of an $\Rb$-valued unary predicate $U_f$ for each continuous function $f: X \to \Rb$, and a $\left(2 + \frac{r}{\diam X} \right)$-Lipschitz $[0,r]$-valued binary predicate $d_r$ for each real $r > 0$. For each continuous $f:X \to \Rb$, we pick a non-decreasing modulus of uniform continuity $\alpha_{U_f}:[0,\diam X]\to \Rb$ such that $f:X \to \Rb$ is $\alpha_{U_f}$-uniformly continuous. For any $L$-Lipschitz function $f$, we will require that $\alpha_{U_f}(x) = Lx$.\footnote{Recall that continuous functions on compact topometric spaces are automatically uniformly continuous with regards to the metric.} The `official' metric of $\Lc_X$ is $d_{\diam X}$.

  We define the \emph{generic $\Rb$-forest-embeddable structure}, written $\Ff(X)$, as the $\Lc_X$-structure whose universe consists of triples $K = (\pi(K),K_X,K_\omega)$ where\footnote{The use of $\omega$ in this definition is arbitrary. Any infinite cardinal would be sufficient.} 
  \begin{itemize}
  \item $\pi(K)$ is a compact subset of $\Rb_{\geq 0}$ containing $0$,
  \item $K_X:\pi(K) \to X$ is a $1$-Lipschitz function, and
  \item $K_\omega:(\pi(K) \setminus \{\sup \pi(K)\}) \to \omega$ is an arbitrary function.
  \end{itemize}

  For $K$ and $L$ in $\Ff(X)$, we say that \emph{$L$ extends $K$}, written $K \sqsupseteq L$, if $\pi(K) \subseteq \pi(L)$, $L_X\res \pi(K) = K_X$, and $L_\omega\res (\pi(K) \setminus \{\sup \pi(K)\}) = K_\omega$. 
 

We will write $|K|$ for $\sup\pi(K)$, which we will call the \emph{length} of $K$.
  
  We say that two elements $K$ and $K'$ of $\Ff(X)$ are in the same \emph{finite distance component} of $\Ff(X)$ (or that they have \emph{finite distance}) if  $K_X(0)=K'_X(0)$.
  
  For any $\Lc_X$-predicate $U_f$, we set
  \[
    U_f^{\Ff(X)}(K) = f(K_X(|K|)).
  \]
  For any $K \in \Ff(X)$ and $r \in [0,|K|]$, we write $K \res [0,r]$ for the unique maximal element of $\Ff(X)$ such that $L \sqsupseteq K$ and $|L| \leq r$. We call elements of this form \emph{initial segments} of $K$.  

  Finally for any $K$ and $K'$, either $K$ and $K'$ are not in the same finite distance component or there is a unique largest $r$ such that $r \in \pi(K)\cap \pi(K')$ and $K\res [0,r] = K' \res [0,r]$. The element $K\res r = K'\res r$ is the \emph{longest common initial segment of $K$ and $K'$}, written $K \sqcap K'$ if it exists.
  
  We define an extended metric $d^{\Ff(X)}$ on $\Ff(X)$ by setting $d^{\Ff(X)}(K,K')$ equal to $\infty$ if $K$ and $K'$ are not in the same finite distance component and $|K| + |K'| - 2|K\sqcap K'|$ otherwise. We may write $d^{\Ff(X)}$ as $d$ if no confusion can arise. For each $s > 0$, we set $d_s^{\Ff(X)} = \min\{d^{\Ff(X)},s\}$.
\end{defn}

Note that if $K \sqsubseteq K'$, then $d(K,K') = |K'| - |K|$.
 
 \begin{prop}
   In any $\Ff(X)$, $d$ is a well-defined extended metric.
 \end{prop}
 \begin{proof}
   For any three $K$, $K'$, and $K''$, if any two of them have infinite distance, then the triangle inequality is clearly satisfied, so assume that $K(0) = K'(0) = K''(0)$. 
   Without loss of generality, assume that $|K\sqcap K'| \leq |K'\sqcap K''|$. 
   We necessarily have that $|K\sqcap K''| \geq |K\sqcap K'|$.

    We now have that
   \[
     d(K,K'') = |K| + |K''| - 2|K\sqcap K''|
   \]
   and
   \[
     d(K,K')+d(K',K'') = |K| +  2|K'| + |K''| - 2|K\sqcap K'| - 2|K'\sqcap K''|,
  \]
  so since $|K'\sqcap K''| \leq |K'|$ and $|K\sqcap K'| \leq |K\sqcap K''|$, we have that $2|K\sqcap K'|+2|K'\sqcap K''| \leq 2|K'| + 2|K\sqcap K''|$ and the triangle inequality holds.

  Finally, $d(x,y)$ is clearly symmetric and satisfies $d(x,y)=0$ if and only if $x=y$.
 \end{proof}

Note the easy fact that $||K| - |K'|| \leq d(K,K')$.

 \begin{lem}
   For any set $\Kbf$ of elements of $\Ff(X)$ with pairwise finite distance, there is a unique longest common initial segment $\bigsqcap \Kbf$ of $\Kbf$.
 \end{lem}
 \begin{proof}
   Let
   \[
     r = \sup\{ s :(\forall K \in \Kbf) s \in \pi(K) \wedge (\forall K,K' \in \Kbf) K \res [0,r] = K' \res [0,r]\}. 
   \]
   By continuity, we have that $(K\res [0,r])_X = (K' \res [0,r])_X$ for any $K,K' \in \Kbf$. Therefore, $K \res [0,r] = K' \res [0,r]$ for any $K,K' \in \Kbf$, and this is the required longest common initial segment.
 \end{proof}


Recall that for any set $A$ in a metric space, the \emph{diameter of $A$}, written $\diam A$, is $\sup\{d(a,a'): a,a' \in A\}$.
 
 \begin{lem}\label{lem:common-init-bound}
   For any set $\mathbf{K}$ of elements of $\Ff(X)$ with pairwise finite distance,
   \[
     \left| \bigsqcap \Kbf \right| \geq \sup_{K \in \Kbf}|K| - \diam \Kbf.
   \]
   \end{lem}
 \begin{proof}
   Fix $K \in \mathbf{K}$. Since all $K' \in \Kbf$ have $d(K,K') \leq \diam \Kbf$, we have
   \begin{align*}
     |K| + |K'| - 2 |K\sqcap K'| &\leq \diam \Kbf, \\
     \frac{1}{2}(|K| + |K'| - \diam \Kbf) &\leq |K \sqcap K'|.
   \end{align*}
   This implies that all elements of $\Kbf$ share with $K$ a common initial segment of length at least
   \[
     \frac{1}{2}(|K| + \inf_{K' \in \Kbf}K' - \diam \Kbf),
   \]
   which means that $\bigsqcap \Kbf$ is at least this long. Taking the supremum over $K \in \Kbf$ gives
\[
   \left| \bigsqcap \mathbf{K} \right| \geq \frac{1}{2}\left( \sup_{K \in \Kbf}| K| + \inf_{K \in \Kbf}| K| - \diam \Kbf\right).
 \]
 Clearly $\inf_{K \in \Kbf}| K| \geq \sup_{K \in \Kbf}| K| - \diam \Kbf$, so we have
 \[
   \left|  \bigsqcap \Kbf \right| \geq \frac{1}{2}\left( 2\sup_{K \in \Kbf} - 2\diam \Kbf \right) = \sup_{K \in \Kbf}| K| - \diam \Kbf.\qedhere
\]
 \end{proof}

 \begin{cor}\label{cor:meet-diam-bound}
   If $\Kbf \subseteq \Ff(X)$ has diameter at most $r < \infty$, then for any $K \in \Kbf$, $d(K,\bigsqcap \Kbf)\leq \diam \Kbf$.
 \end{cor}
 \begin{proof}
   For any $K \in \Kbf$, we have that
   \[
     d\left(K,\bigsqcap \Kbf\right) = | K| - \left|  \bigsqcap \Kbf \right| \leq | K| +\diam \Kbf -\sup_{K \in \Kbf} | K| \leq \diam \Kbf,
   \]
   as required.
 \end{proof}
 
 \begin{prop}
   The metric $d$ on $\Ff(X)$ is complete.
 \end{prop}
 \begin{proof}
   Let $\{K_i\}_{i<\omega}$ be a Cauchy sequence in $\Ff(X)$. By passing to a final segment, we may assume that the elements of this sequence have pairwise finite distance. Let $L_j = \bigsqcap\{ K _i : i \geq j\}$. It is clear that $L_j$ is an increasing sequence in the sense that $L_{j+1}\sqsupseteq L_j$ for every $j < \omega$. Furthermore, by \cref{cor:meet-diam-bound}, we have that $d(K_i,L_i) \to 0$ as $i \to \infty$.

   Let $A = \bigcup_{i<\omega} L_i$. Either $|A| = 0$, in which case do nothing, or $|A| > 0$, in which case there is a unique pair $(|A|,x)$ with $x \in X$ which makes $B = A\cup\{(|A|,x)\}$ an element of $\Ff(X)$. 
   By construction we have that $B$ is the limit of the sequence $\{L_i\}_{i<\omega}$ and therefore of the sequence $\{K_i\}_{i<\omega}$ as well.
\end{proof}

\begin{prop}\label{prop:unif-cont}
  For any continuous function $f: X \to \Rb$, the interpretation $U_f^{\Ff(X)}$ is $\alpha_{U_f}$-uniformly continuous.
\end{prop}
\begin{proof}
  Recall that we have chosen $\alpha_{U_f}$ so that $f$ is $\alpha_{U_f}$-uniformly continuous on $X$. Also, note that we are really talking about uniform continuity with regards to the `official' metric $d_{\diam X}$.

  Fix $ K,K' \in \Ff(X)$. If $d(K,K')\geq \diam X$, then there is nothing to prove, so assume that $d(K,K') < \diam X$. Since the induced functions $K_X$ and $K'_X$ are $1$-Lipschitz, we have that
  \begin{align*}
    \der(K_X(|K|),K'_X(|K'|)) &\leq  \der(K_X(|K|),K_X(|K\sqcap K'|))+\der(K_X'(|K\sqcap K'|),K'_X(|K'|)) \\
                          & \leq d(K,K\sqcap K') + d(K\sqcap K', K') \\
    &  = d(K,K')
  \end{align*}
  Therefore
  \begin{align*}
    |U_f(K)-U_f(K')| &\leq \alpha_{U_f}(\der(K_X(|K|),K'_X(|K|))) \\
                     &\leq \alpha_{U_f}(d(K,K')). 
  \end{align*}
  So since $\alpha_{U_f}$ is non-decreasing, $U_f$ is $\alpha_{U_f}$-uniformly continuous in $\Ff(X)$.
\end{proof}

\begin{cor}
  $\Ff(X)$ is an $\Lc_X$-structure.
\end{cor}
\begin{proof}
  Given \cref{prop:unif-cont}, the only thing to verify is that the predicate symbols $d_r$ obey the correct moduli of uniform continuity relative to the `official' metric $d_{\diam X}$. For any $K$, $K'$, $L$, and $L'$, we have that
  \begin{align*}
    |d_r(K,L)-d_r(K',L')| &\leq \min\{|d(K,L)-d(K',L')|,r\}\\
                          &\leq \min\{2\max\{d(K,K'),d(L,L')\},r\}\\
                          & \leq \left( 2 + \frac{r}{\diam X} \right) \min\{\max\{d(K,K'),d(L,L')\}, \diam X\} \\
                          & = \left( 2 + \frac{r}{\diam X} \right) \max\{d_{\diam X}(K,K'),d_{\diam X}(L,L')\}.\qedhere
  \end{align*}
\end{proof}

\begin{defn}
  For any $X$, let $T(X)$ be $\Th(\Ff(X))$.
\end{defn}

\section{The Space of Paths}

\noindent In this section we develop some machinery needed to analyze arbitrary models of $T(X)$. 

\begin{defn}
  For any compact topometric space $(X,\tau,\der)$, let the \emph{set of paths in $X$}, written $\Pf(X)$, be the collection of all partial $1$-Lipschitz functions $f:{\subseteq}\Rb_{\geq 0} \to X$ with compact domain containing $0$. We write $|f|$ for $\sup\dom f$. 

\end{defn}

Note that for any $K \in \Ff(X)$, we have that $K_X$ is an element of $\Pf(X)$. Also note that elements of $\Pf(X)$ are automatically \emph{topologically} continuous as well, since the ordinary topology is coarser than the metric topology.

We put a uniform structure (and therefore a topology) on $\Pf(X)$ generated by
entourages of the form $U_{V,\e}$ for $V$, an entourage in $X^2$, and $\e > 0$, where $(f,g) \in U_{V,\e}$ if and only if
\begin{itemize}
\item for every $r \in \dom f$, there is an $s \in \dom g$ such that $(f(r),g(s)) \in V$ and $|r-s| < \e$ and
\item for every $s \in \dom g$, there is an $r \in \dom f$ such that $(f(r),g(s)) \in V$ and $|r-s| < \e$.
\end{itemize}
To see that this generates a uniform structure on $\Pf(X)$, note that
\begin{itemize}
\item $U_{V\cap W, \min\{\e,\delta\}} \subseteq U_{V,\e}\cap U_{W,\delta}$ and
\item if $W^{\circ 2} \coloneqq \{(x,z) : (\exists y \in X) (x,y) \in W \wedge (y,z) \in W\} \subseteq V$, then $U_{W,\e/2}^{\circ 2} \subseteq U_{V,\e}$.
\end{itemize}

Recall that a uniform structure is \emph{complete} if every Cauchy net converges, where a \emph{Cauchy net} is a net $\{x_i\}_{i\in I}$ such that for every entourage $V$, there is an $i \in I$ such that $(x_j,x_k) \in V$ for all $j,k\geq i$. A uniform structure is \emph{Hausdorff} if the induced topology is Hausdorff. For any entourage $V$, we write $V(x)$ for the set $\{y \in X : (x,y) \in V\}$.


\begin{prop}\label{prop:top-convergence-char}
  The uniform structure on $\Pf(X)$ is Hausdorff and complete.
\end{prop}
\begin{proof}
  First to see that the uniform structure on $\Pf(X)$ is Hausdorff, let $f$ and $g$ be distinct elements of $\Pf(X)$. If $\dom f \neq \dom g$, then there must be an $\e > 0$ small enough that $(f,g) \notin U_{V,\e}$ for any entourage $V$. If $\dom f \neq \dom g$, then there must be some $r \in \dom f$ such that $f(r) \neq g(r)$. Find an entourage $V$ small enough that $g(r) \notin \cl V(f(r))$, and then find $\e > 0$ small enough that $g(r) \notin (\cl V(f(r)))^{\leq \e}$. Now assume that $(f,g) \in U_{V,\e}$. By definition, this means that there is some $s \in \dom g$ with $|r-s| < \e$ such that $(f(r),g(s)) \in V$, i.e., $g(s) \in V(f(r))$. Since $g$ is $1$-Lipschitz, we have that $\der(g(r),g(s)) \leq |r-s| < \e$, which implies that $g(r) \in V(f(r))^{<\e} \subseteq (\cl V(f(r)))^{\leq \e}$, which is a contradiction. Therefore $(f,g) \notin U_{V,\e}$.

  To show that the uniform structure is complete, let $\{f_i\}_{i \in I}$ be a net on some directed set $I$.

  Let $F$ be the set of points $r$ in $\Rb_{\geq 0}$ with the property that for every $\e > 0$, there is an $i \in I$ such that for all $j \geq i$, $r$ has distance at most $\e$ from the domain of $f_j$. It is clear that $0 \in F$ and that $F$ is closed. By looking at $f_i$ for some sufficiently large $i \in I$, we can see that $F$ must be bounded and therefore compact.

  For each $r \in F$, define a net $\{x_i^r\}_{i \in I}$ of points in $X$ by setting $x_i^r$ to $f_i(s_i^r)$ where $s_i^r$ is the smaller of the (one or two) nearest points in $\dom f_i$ to $r$. (Note that this is well defined since $\dom f_i$ is always non-empty.)

  \emph{Claim}. For each $r \in F$, the net $\{x_i^r\}_{i \in I}$ is convergent.

  \emph{Proof of claim.} Fix an entourage $V\subseteq X^2$. Find an entourage $W$ and an $\e >0$ small enough that the set
  \[
   A\coloneqq \{(x,y) \in X^2 : \exists z(x,z)\in \cl W\wedge \der(z,y) < 3\e \}
  \]
  is contained in $V$. (This is always possible by compactness.) Now find $i \in I$ large enough that for any $j,k \geq i$, $f_j \in U_{f_i,W,\e}$ and the distance between $r$ and the domain of $f_j$ and $f_k$ is at most $\e$. This implies that for any $j,k\geq i$, there is some $t \in \dom f_j$ such that $|s_j^r - t| < \e$ and $(x_j^r,f_k(t)) \in W$. Since $|s_j^r-r| \leq \e$, we have that $|r-t| < 2\e$. Likewise, $|s_k^r-r| \leq \e$, so $|t-s_k^r| < 3\e$.  This implies that $\der(f_k(t),x_k^r) < 3\e$. Therefore $f_k(t)$ witnesses that $(x_j^r,x_k^r)$ is in $A$ and therefore also $V$.

  Since we can do this for any entourage $W$, we have that $\{x_i^r\}_{i \in I}$ is a convergent net. \hfill $\square_{\text{claim}}$

  Let $g(r)$ be the unique limit point of the net $\{x_i^r\}_{i \in I}$ for each $r \in F$. By lower semi-continuity of $\der$, $g(r)$ must be $1$-Lipschitz, so it is an element of $\Pf(X)$ and the limit of the net $\{f_i\}_{i \in I}$.
\end{proof}

Note that the converse of \cref{prop:top-convergence-char} is immediate, so we have a characterization of topological convergence in $\Pf(X)$. 

\begin{prop}\label{prop:top-is-ultra}
  For any $r \geq 0$, there is a theory $T_r$ in the language $\Lc_X(c)$ (i.e., $\Lc_X$ with a single constant added) such that the models of $T_r$ can be naturally identified with the elements of $\Pf_r(X) \coloneqq \{ f \in \Pf(X) : |f| \leq r\}$ such that the topology on $\Pf_r(X)$ agrees with the topology induced by ultraproducts. In particular, each $\Pf_r(X)$ is compact.
\end{prop}
\begin{proof}
 Consider the theory $T_r$ containing $d_{r+1}(x,y) \leq r$ and
  \[
    d_{r}(x,y)+d_{r}(x,z)+d_{r}(y,z) = 2 \max\{d_{r}(x,y),d_{r}(x,z),d_{r}(y,z)\}
  \]
  for all $x$, $y$, and $z$ 
  and
  \[
     \min\{d(c,x),d(c,y)\}+ d(x,y) = \max\{d(c,x),d(c,y)\}
   \]
   for all $x$ and $y$ (as well as axioms establishing the relationship between $d_s$ for various $s$). The first axiom ensures $\Rb$-embeddability and the second that $c$ is an endpoint of a model of this structure into $\Rb$, which we can take to be $0$. Any such structure must have that the embedding into $X$ is $1$-Lipschitz by our requirement that $\alpha_{U_h}(x) = Lx$ whenever $h: X \to \Rb$ is an $L$-Lipschitz function (by \cite[Thm.\ 1.6]{BYTopo2010}), so we have that $\Pf_r(X)$ corresponds precisely to the models of this theory.

  Now we need to show that the topology on $\Pf_r(X)$ agrees with the topology induced by taking ultraproducts. It is clear that if a net of elements of of $\Pf_r(X)$ converges in the topology, then any ultraproduct with an ultrafilter extending the net will converge to the same structure, so we have that the ultraproduct topology is coarser than the topology on $\Pf_r(X)$. Now consider a family $\{f_i\}_{i \in I}$ of elements of $\Pf_r(X)$, thought of as $\Lc_X(c)$-structures, together with an ultrafilter $\Uc$ on $I$. Let the ultraproduct be $f_{\Uc}$. We know that this corresponds to an element of $\Pf_r(X)$, and we will identify it with this corresponding element. 

  For each $s \in \dom f_{\Uc}$ and any finite sequence $h_1,\dots,h_n$ of continuous $\Rb$-valued functions on $X$, we must have that
  \[
    f_{\Uc} \models \inf_{x}\max\{|d(x,c)-s|,\max_{\ell\leq n} |U_h(x) - h_\ell(f_{\Uc}(s))|\} = 0.
  \]
  Therefore, it must be the case that for any $\e > 0$, there is a $\Uc$-large set of indices such that $\dom f_i$ contains an element $t$ with distances less than $\e$ from $s$ such that $|h_\ell(f_i(t)) - h_\ell(f_{\Uc}(s))| < \e$ for every $i \leq \ell$. Since we can do this for any $\e > 0$ and any finite sequence of continuous $\Rb$-valued functions on $X$, we have by \cref{prop:top-convergence-char} that the family $\{f_i\}_{i \in I}$ converges along the ultrafilter $\Uc$ in the ordinary topology on $\Pf_r(X)$.

  Therefore the topologies agree, and $\Pf_r(X)$ is compact.
\end{proof}

\section{First-Order Theory of $\Ff(X)$}

\noindent In any model $M$ of $T(X)$, we define an extended metric $d$ by setting $d(x,y) = \sup_r d_r(x,y)$. The theory of $\Ff(X)$ ensures that this is an extended metric satisfying $d_r(x,y) = \min\{d(x,y),r\}$ for every $r > 0$.

\begin{prop}\label{prop:interval-definable}
  For any $r> 0$ and  $K$ and $K'$ in $\Ff(X)$ with $d(K,K') < r$, the formula
  \[
\delta_{K,K',r}(x) \coloneqq \min\{d_{2r}(x,K) + d_{2r}(x,K') - d_r(K,K'), r\} 
\]
is the distance predicate (with regards to $d_s$ for any $s \geq r$) of the set of $L$ that satisfy either $K\sqcap K' \sqsubseteq L \sqsubseteq K$ or $K\sqcap K' \sqsubseteq L \sqsubseteq K'$.
\end{prop}
\begin{proof}
  Let $[K,K']$ denote the set described in the proposition.
  
  Let $A$ be an element in the same finite distance component as $K$ and $K'$. Assume without loss of generality that $A\sqcap K$ is longer than $A \sqcap K'$. Let $B = A \sqcap K$.

  If $|B| < |K \sqcap K'|$, then we have that
  \[
    d(B,K \sqcap K') = d(B,K) + d(B,K') - d(K,K')
  \]
  and it is easy to check that $d(A,C) \geq d(A,K \sqcap K') = d(A,B) + d(B,K\sqcap K')$ for any $C \in [K,K']$. This implies that $d(A,[K,K']) = d(A,K)+d(A,K') - d(K,K')$.

  If $|B| \geq |K\sqcap K'|$, then it must be the case that $K \sqsupseteq B \sqsupseteq K\sqcap K'$, which implies that $d(A,[K,K']) = d(A,B)$ and so also $d(A,[K,K']) = d(A,K)+d(A,K') - d(K,K')$.

  So in either case we have that $d(A,[K,K']) = d(A,K)+d(A,K') - d(K,K')$. It is straightforward to check that the formula in the proposition is equal to $\min\{d(A,[K,K']), r\}$. 
\end{proof}



\begin{defn}
  For any $K$ and $K'$ in the same finite distance component, we will write $[K,K']$ for the set from \cref{prop:interval-definable}, called the \emph{interval between $K$ and $K'$}.

  We say that two intervals $[K,K']$ and $[L,L']$ are \emph{isomorphic} if they correspond to the same element of $\Pf(X)$ with $K$ and $L$ as basepoints.
\end{defn}

Note that $[K,K']$ and $[L,L']$ are isomorphic if and only if they are isomorphic as $\Lc_X$-structures by an isomorphism taking $K$ to $L$.

Since for each $r > 0$, $[K,K']$ is uniformly definable for $K$ and $K'$ with $d(K,K') \leq r$, the first-order theory of the structure $\Ff(X)$ ensures that similar sets exist in any $M \equiv \Ff(X)$.

\begin{cor}
  For any $M \equiv \Ff(X)$ and any $K$ and $K'$ in $M$ with $d(K,K') < r$, the formula $\delta_{K,K',r}(x)$ is the distance predicate (with regards to the metric $d_s$ for any $s \geq r$) of a set that is isometric to a closed subset of $[0,d(K,K')]$ with $K$ and $K'$ as endpoints.
\end{cor}
\begin{proof}
  This is a first-order property; specifically, $[K,K']$ with the constant $c$ assigned to $K$ is a model of $T_r$, the theory of elements of $\Pf_r(X)$. 
\end{proof}


\begin{prop}\label{prop:interval-projection}
  If $M$ is any model of $T(X)$, then for any $K,K',L \in M$ with pairwise finite distance, there is a unique point $A \in [K,K']$ such that $d(L,A) = d(L,[K,K'])$.
\end{prop}
\begin{proof}
  We will show that in $\Ff(X)$, for any $K, K', L$ with pairwise distance $<r$, the formula $d(x,L) - d(L,[K,K'])$ is, inside $[K,K']$, the distance predicate of a singleton. This property is preserved under ultraproducts for every $r > 0$, so the required statement will follow.
  
  Fix $K,K',L \in \Ff(X)$ with pairwise finite distance. There are two cases: 

  \emph{Case 1.} Either $K \sqcap L \in [K,K']$ or $K'\sqcap L \in [K,K']$. Note that either possibility implies that $L \sqsupseteq K \sqcap K'$, so either possibility implies the other. Assume without loss of generality that $|K\sqcap L| \geq |K' \sqcap L|$. This implies that $K' \sqcap L = K\sqcap K'$. We have immediately that $d(L,[K,K']) \leq d(L,K\sqcap L)$. For any element $A$ of $[K,K\sqcap K']$, we clearly have that $d(L,A) = d(L,K\sqcap L) + d(K\sqcap L,A)$, so $d(K\sqcap L,A) = d(L,A) - D(L,K\sqcap L)$. For any element $B$ of $[K',K\sqcap K']$, we have that $d(L,B) = d(L,K\sqcap L) + d(K\sqcap L, K\sqcap K') + d(K\sqcap K', B)$, so $d(B,K\sqcap L) = d(B, K\sqcap K') + d(K \sqcap K',K\sqcap L) = d(L,B) - d(L,K\sqcap L)$. In either case, we have that for any $C \in [K,K']$, $d(C,K\sqcap L) = d(L,C) - d(L,K\sqcap L)$, as required.

  \emph{Case 2.} Neither $K \sqcap L \notin [K,K']$ nor $K' \sqcap L \in [K,K']$. This implies that $K \sqcap L \sqsubset K\sqcap K'$, which means that $K \sqcap L = K' \sqcap L$. Therefore, for any $A \in [K,K']$, we have that $d(L,A) = d(L,K\sqcap K') + d(K \sqcap K', A)$, whence $d(A,K\sqcap K') = d(L,A) - d(L,K\sqcap K')$, as required.
\end{proof}

Note, though, that the map taking $L$ to the nearest point $A$ in $[K,K']$ cannot be a definable function, since it is only well defined inside the finite distance component of $[K,K']$, which is a co-type-definable set. It is, however, representable as a family of partial definable functions with domains containing $[K,K']^{<r}$ for each $r > 0$.

\begin{defn}
  In any model $M$ of $T(X)$, a \emph{finite tree} is a set which can be written as a union of a finite sequence $\{[K_i,L_i]\}_{i < n}$ with the property that for each $i< n$ with $i > 0$, $[K_i,L_i]$ is not disjoint from $\bigcup_{j< i}[K_j,L_j]$.
\end{defn}

Since a finite tree is a finite union of definable sets, it is itself always a definable set. Note also that all elements of a finite tree are automatically in the same finite distance component.

\begin{lem}\label{lem:tree-contain-interval}
  If $R$ is a finite tree, then for any $K,L \in R$, $[K,L] \subseteq R$.
\end{lem}
\begin{proof}
  We will proceed by induction. First assume that $R$ is $[A,B]$. In $\Ff(X)$, for any given $r > 0$, we have for all $A,B$ with $d(A,B) < r$, that any $K,L \in [A,B]$ has $[A,B] \subseteq [K,L]$. Since $[K,L]$ and $[A,B]$ are uniformly definable (as long as $d(K,L) < r$ and $d(A,B) < r$), this is a first-order fact and will still be true in any model of $T(X)$.

  Assume that the assertion is true for all finite trees that are unions of $n\geq 1$ intervals. Let $R$ be a finite tree that is the union of $n$ intervals, and let $[A,B]$ be some interval with $R \cap [A,B]$ non-empty. If $K,L \in R$ or $K,L \in [A,B]$, then we have the assertion immediately, so assume without loss of generality that $K \in R$ and $L \in [A,B]$. Let $C$ be some element of $R \cap [A,B]$. By the induction hypothesis, we have that $[K,C] \subseteq R$ and $[C,L]\subseteq [A,B]$.

  In $\Ff(X)$, it is easy to check that for any three elements $E$, $F$, and $G$ in the same finite distance component, $[E,G] \subseteq [E,F] \cup [F,G]$. For any $r > 0$, this is a first-order fact for $E$, $F$, and $G$ with pairwise distance less than $r$, so it will hold in any model of $T(X)$ as well.

  This implies that $[K,L] \subseteq [K,C] \cup [C,L] \subseteq R \cup [A,B]$.

  Therefore, by induction, we have that the result holds for all finite trees.
\end{proof}

\begin{cor}\label{cor:convex-closure}
  For any $M \models T(X)$ and any finite $n$-tuple $\bar{a} \in M$ with pairwise finite distances, $R=\{a_0\}\cup[a_0,a_1]\cup\dots\cup [a_{n-2},a_{n-1}]$ is the intersection of all finite trees containing $\bar{a}$.
\end{cor}
\begin{proof}
  Clearly $R$ is a finite tree, so the intersection of all finite trees containing $\bar{a}$ must be a superset of $R$. By \cref{lem:tree-contain-interval}, $R$ must be a subset of any finite tree containing $\bar{a}$, so we are done.
\end{proof}

\begin{defn}
  For any model $M \models T(X)$ and any tuple $\bar{a} \in M$ with pairwise finite distance, the \emph{convex closure of $\bar{a}$}, written $\ccl(\bar{a})$, is the intersection of all finite trees containing $\bar{a}$.
\end{defn}

We will not need to prove this, but for finite tuples $\bar{a}$ with pairwise finite distance, $\ccl(\bar{a})$ is actually both the definable and algebraic closures of $\bar{a}$.

\begin{prop}\label{prop:tree-proj}
  For any model $M$ of $T(\Ff(X))$, any finite tree $R$ in $M$, and any $K \in M$ in the same finite distance component of $M$, there is a unique element $L \in R$ with $d(K,L)=d(K,R)$.
\end{prop}
\begin{proof}
  Proceed by induction on the number of intervals in the finite tree. If the finite tree is a single interval, then by \cref{prop:interval-projection} the result follows.

  Suppose that we know the result for all finite trees that are unions of $n$ intervals. Let $R$ be a finite tree that is the union of $n$ intervals and let $[A,B]$ be some interval with $R \cap [A,B]$ non-empty. Clearly $d(K,R\cup [A,B]) = \min\{d(K,R),\allowbreak d(K,[A,B])\}$. By the induction hypothesis, there is a unique point $L_0 \in R$ with $d(K,L_0) = d( K, R )$ and a unique point $L_1 \in [A,B]$ with $d(K,L_1) = d(K,[A,B])$. If $d(K,L_0) \neq d(K,L_1)$, we are done, so assume that $d(K,L_0)=d(K,L_1)$. Again by the induction hypothesis, there is some $L_2 \in [L_0,L_1]$ such that $d(K,L_2) = d(K,[L_0,L_1])$. If $L_0 \neq L_1$, then we must have, by the uniqueness of $L_2$, that $d(K,L_2) < d(K,L_0) = d(K,L_1)$, but by \cref{lem:tree-contain-interval}, we have that $L_2 \in R \cup [A,B]$. This implies that either $d(K,R) < d(K,L_0)$ or $d(K,[A,B]) < d(K,L_1)$, which is absurd. Therefore there must be some unique $L$ in $R\cup [A,B]$ such that $d(K,L) = d(K,R\cup [A,B])$.

  Therefore, by induction, we have that the assertion is true for all finite trees.
\end{proof}

\begin{prop}
  For any ultrafilter $\Uc$ on index set $I$, if $\{a_i\}_{i \in I}$ and $\{b_i\}_{i \in I}$ are families of elements of $\Ff(X)$ whose corresponding elements $a$ and $b$ in $\Ff(X)^\Uc$ satisfy $d(a,b) < \infty$, then $[a,b]$ in the limit corresponds to the element of $\Pf(X)$ that is the limit along $\Uc$ of the elements in $\Pf(X)$ corresponding to the $[a_i,b_i]$'s.
\end{prop}
\begin{proof}
  This is immediate from Propositions~\ref{prop:top-is-ultra} and \ref{prop:interval-definable} (as well as the fact that ultraproducts commute with restricting to definable sets).
\end{proof}

Now we come to the first point at which we actually need to assume that the metric $\der$ on $X$ is open.

\begin{lem}[Parallel Paths]\label{lem:par-paths}
($\der$ open.) For any $f \in \Pf(X)$, any entourage $V$, and any $\e > 0$, there is an open neighborhood $O \ni f(0)$ such that for any $x \in O$, there is a path $g \in \Pf(X)$ such that $g(0) = x$ and $(f,g) \in U_{V,\e}$. 
\end{lem}
\begin{proof}
  For any entourage $W \subseteq X^2$ and any $z \in X$, let $W(z) = \{w: (w,z) \in W\}$. Obviously $W(z)$ is an open neighborhood of $z$.

  Find an entourage $W \subseteq V$ and a $\delta > 0$ with $\delta<\e$ small enough that for any $x \in X$, $\cl(W(x))^{\leq \delta} \subseteq V(x)$. (This is always possible by compactness.)

  Let $F$ be a finite subset of $\dom f$ such that $0 \in F$ and $d_H(F,\dom f) < \frac{1}{2}\delta$. (Such a set always exists by compactness.) Let $\{r_i\}_{i \leq n}$ be an increasing enumeration of $F$ (with $r_0 = 0$). Fix $\gamma > 1$ small enough that  $(\gamma-1)|f| < \frac{1}{2}\delta$. 
  For each $i <n$, let $d(i) = d(f(r_i),f(r_{i+1}))$.

  Let $A_i = W(f(r_i))$ for each $i \leq n$. For each $i <n$,
  \begin{itemize}
  \item if $f(r_i) = f(r_{i+1})$, let $B_i$ and $C_{i+1}$ both equal $X$ and
  \item if $f(r_i) \neq f(r_{i+1})$, find open neighborhoods $B_i \ni f(r_i)$ and $C_{i+1} \ni f(r_{i+1})$ such that $\cl(B_i)^{\leq \gamma^{-1}d(i)} \cap \cl(C_{i+1}) = \varnothing$. (This is always possible in any compact topometric space.)
  \end{itemize}
  Let $D_0 = A_0 \cap B_0$, $D_i = A_i \cap B_i \cap C_i$ for each $0<i<n$, and $D_n = A_n \cap C_n$. Note that each $D_i$ is an open neighborhood of $f(r_i)$. What we have guaranteed at this point is that for each $i< n$, if $y \in D_i$ and $z \in D_{i+1}$, then $\der(y,z) \geq \gamma^{-1}d(i)$.

  Let $E_n = D_n$. For each $i < n$, let $E_i = D_i \cap E_{i+1}^{< \gamma d(i)}$. Note that each $E_i$ is an open neighborhood of $f(r_i)$ (since $d$ is an open metric). Now note that for each $i<n$, for any $y \in E_i$, there is a $z \in E_{i+1}$ such that $\der(y,z) < \gamma d(i)$. Since $y$ and $z$ are also in $D_i$ and $D_{i+1}$, respectively, we also have that $\der(y,z) \geq \gamma^{-1}d(i)$.

  So let $O = E_0$. For any $x \in O$, by construction, we can find a sequence $\{x_i\}_{i\leq n}$ such that
  \begin{itemize}
  \item $x_0 = x$,
  \item $x_i \in E_i \subseteq W(f(r_i))$ for each $i \leq n$, and
  \item $\gamma^{-1} d(i) \leq \der(x_i,x_{i+1}) \leq \gamma d(i)$ for each $i < n$.
  \end{itemize}
  We are note quite done, as it may be the case that the function that maps $r_i$ to $x_i$ is not $1$-Lipschitz. What we do have is that for each $i < n$, $\gamma^{-1}\der(x_i,x_{i+1}) \leq \der(f(r_i),f(r_{i+1})) \leq |r_{i+1}-r_i|$, so $\der(x_i,x_{i+1}) \leq \gamma |r_{i+1}-r_{i}|$, which implies that the function that maps $\gamma r_i$ to $x_i$ is $1$-Lipschitz.

  Let $g$ be the element of $\Pf(X)$ with domain $\gamma F$ with the property that for each $i \leq n$, $g(\gamma r_i) = x_i$. We want to show that $g \in U_{f,V,\e}$. 

  For each $\gamma r_i$ in $\dom g$, we have by construction that $|\gamma r_i - r_i| \leq (\gamma-1)|f| < \frac{1}{2}\delta < \e$, and furthermore we have that $g(\gamma r_i) \in W(f(r_i)) \subseteq V(f(r_i))$.

  For the other direction, we have that for any $s \in \dom f$, there is an $r_i \in F$ with $|s-r_i| < \frac{1}{2}\delta$. This implies that $|s-\gamma r_i | < \delta < \e$. By construction, we have that $f(r_i) \in W(g(\gamma r_i))$. Since $|s-r_i| < \frac{1}{2}\delta$, we have that $\der(f(s),f(r_i)) < \frac{1}{2}\delta < \delta$ as well, so $f(s) \in (\cl W(g(\gamma r_i)))^{\leq \delta} \subseteq V(g(\gamma r_i))$, as required.

  Therefore $g \in U_{f,V,\e}$. Finally, $|g| \leq \gamma|f| < |f| + \e$.
\end{proof}


\begin{defn}
  For any model $M\models T(X)$ and any $a \in M$, we write $\tp_X(a)$ for the unique $x \in X$ with the property that for every continuous $f : X \to \Rb$, $f(x) = U_f^M(a)$.
\end{defn}

Note that $\tp_X(a)$ is essentially the quantifier-free type of $a$. The notation is mostly to emphasize that we are thinking of it as an element of $X$.

\begin{lem}\label{lem:independent-extensions}
  ($\der$ open.) For any model $M \models T(X)$, any $b \in M$, any $f \in \Pf(X)$ with $f(0) = 0$, and any $\kappa$, there is an elementary extension $N \succeq M$ and a family $\{c_i\}_{i<\kappa}$ of elements of $N$ such that for any $i<\kappa$, $[b,c_i]$ exists and is isomorphic to $f$ and such that for any $i<j<\kappa$, $[b,c_i]\cap[b,c_j] = \{b\}$.
\end{lem}
\begin{proof}
  Clearly by compactness it is sufficient to show this with $\kappa=\omega$. Fix $M \models T(X)$, $b \in M$, and $f \in \Pf(X)$. Since $b$ is an element of a model of $T(X)=\Th(\Ff(X))$, there exists an ultrafilter $\Fc$ and an $a \in \Ff(X)^\Fc$ such that $a \equiv b$. Let $a$ correspond to the family $\{K_i\}_{i\in I}$, where $I$ is the index set of $\Fc$.

  \cref{lem:par-paths} implies that for each entourage $V$ and $\e > 0$, we can find an open neighborhood $O \ni f(0)$ such that for any $x \in O$, there is a path $g \in \Pf(X)$ such that $g(0) = x$ and $(f,g) \in U_{V,\e}$. This implies that for a $\Fc$-large set of $i$, we can find $L_i^n \in \Ff(X)$ for each $n<\omega$ such that
  \begin{itemize}
  \item $[K_i,L_i^n]$ corresponds to some $g$ in $\Pf(X)$ with $g(0)=(K_i)_X(0)$ and $(f,g) \in U_{V,\e}$ for every $n<\omega$ and
  \item for any $n<k<\omega$, $[K_i,L_i^n]\cap [K_i,L_i^k] = \{K_i\}$. 
  \end{itemize}

  Since we can do this for any entourage $V$ and $\e > 0$, we have, by compactness, that there is an elementary extension $C \succ \Ff(X)$ and a family $\{B^n\}_{n<\omega}$ of elements of $C$ such that
  \begin{itemize}

  \item $[a,B^n]$ corresponds to $f \in \Pf(X)$ for every $n<\omega$ and
  \item for any $n<k<\omega$, $[a,B^n]\cap [a,B^k]=\{K_i\}$.
  \end{itemize}
  Since $b \equiv a$, the required elementary extension of $N$ must exist as well. 
\end{proof}

\begin{lem}\label{lem:type-extension-lemma}
  ($\der$ open.) Let $\Uc$ be the monster model of $T(X)$. Let $\bar{a} = \bar{a}^0\bar{a}^1\dots \bar{a}^{n-1}$ and $\bar{b} = \bar{b}^0\bar{b}^1\dots\bar{b}^{n-1}$ be tuples of elements of $\Uc$ partitioned into finite distance classes. Let $A = \ccl(\bar{a}^0)\cup \ccl(\bar{a}^1)\cup \dots \cup \ccl(\bar{a}^{n-1})$ and $B = \ccl(\bar{b}^0)\cup \ccl(\bar{b}^1)\cup \dots \cup \ccl(\bar{b}^{n-1})$. Assume that there is an $\Lc_X$-isomorphism $f: A \cong B$ such that for each $i,j$, $f(a_i^j)=b_i^j$. Then for any $c \in \Uc$, there exists an $e \in \Uc$ such that
  \begin{itemize}
  \item if $c$ is not in the finite distance class of any element of $\bar{a}$, then $e$ is not in the finite distance class of any element of $\bar{b}$, and the map $g:Ac\to Be$ extending $f$ by letting $g(c) = e$ is an $\Lc_X$-isomorphism and
  \item if $c$ is in the finite distance class of $\bar{a}^i$, then $e$ is in the finite distance class of $\bar{b}^i$, and there is an $\Lc_X$-isomorphism $g:A \cup \ccl(\bar{a}^ic) \cong B\cup \ccl(\bar{b}^ie)$ extending $f$ such that $g(c) = e$.
  \end{itemize}
\end{lem}
\begin{proof}
  If $c$ is not in the same finite distance class as any element of $\bar{a}$, then we can easily find $e \in \Uc$ not in the same finite distance class as any element of $\bar{b}$ such that $\tp_X(e) = \tp_X(c)$. Then $g$ extending $f$ to $Ac$ in the obvious way is clearly an $\Lc_X$-isomorphism.

  If $c$ is in the same finite distance class as $\bar{a}^i$, then by \cref{prop:tree-proj}, there is a unique element $c' \in \ccl(\bar{a}^i)$ with $d(c,c') = d(c,\ccl(\bar{a}^i))$. Let $h$ be the element of $\Pf(X)$ corresponding to $[c',c]$. By assumption, we have that $e' \coloneqq f(c') \in B$ has $\tp_X(e')=\tp_X(c')$, so by \cref{lem:independent-extensions}, there is a family $\{e_i\}_{i<(2^{\aleph_0})^+}$ of elements of $\Uc$ such that for each $i<(2^{\aleph_0})^+$, $[e',e_i]$ corresponds to $h$ in $\Pf(X)$ and for each $i<j<(2^{\aleph_0})^+$, $[e',e_i]\cap [e',e_j] = \{e'\}$. Since the cardinality of $\ccl(\bar{b}^i)$ is at most $2^{\aleph_0}$, by the pigeonhole principle, there must be some $i<(2^{\aleph_0})^+$ such that $\ccl(\bar{b}^i)\cap [e',e_i] = \{e'\}$. Let $e$ be that $e_i$.

  We can extend $f$ to $g$ by setting $g(x)$, for each $x \in [c',c]\setminus\{c'\}$, to the unique element $y$ of $[e',e]\setminus \{e'\}$  such that $d(c',x) = d(e',y)$. Since $[c',c]$ and $[e',e]$ both correspond to $h$ in $\Pf(X)$, we have that $g$ is an $\Lc_X$-isomorphism. Finally, we clearly have that $g(c) = e$.
 \end{proof}

 \begin{prop}\label{prop:finite-type-characterization}
   ($\der$ open.) For any finite tuple $\bar{a}$ in any model $M \models T(X)$, $\tp(\bar{a})$ is uniquely determined by the partitioning of $\bar{a}$ into finite distances classes and the $\Lc_X(\bar{b})$-isomorphism type of each $\ccl(\bar{b})$ for $\bar{b}$, a finite distance class of $\bar{a}$.
 \end{prop}
 \begin{proof}
   This follows from \cref{lem:type-extension-lemma} and a back-and-forth argument.
 \end{proof}

 \begin{cor}\label{cor:singleton-type-characterization}
   ($\der$ open.) For any $a \in M \models T(X)$, $\tp(a)$ is uniquely determined by $\tp_X(a)$.
 \end{cor}
 \begin{proof}
   Clearly we have that if $\tp_X(a) \neq \tp_X(b)$, then $\tp(a) \neq \tp(b)$. Conversely, if $\tp_X(a) = \tp_X(b)$, then by \cref{prop:finite-type-characterization}, we have that $\tp(a) = \tp(b)$.
 \end{proof}

\section{Stability of $T(X)$}
\noindent From now on we will assume that $\der$ is an open metric.

\begin{lem}\label{lem:project-onto-model}
  For any model $M \models T(X)$, any elementary extension $N \succeq M$, and any $a \in N$, either $a$ does not have finite distance to any element of $M$ or there is a unique $e \in M$ with minimal distance to $a$. 
\end{lem}
\begin{proof}
  Consider the set $F\coloneqq \bigcap \{ [a,c]^N : c \in M\text{, }d(a,c) < \infty\}$. Since this is the intersection of a family of compact sets, it itself is compact. So in particular, it contains an element $e$ such that $d(e, M)$ is minimized.

  \emph{Claim.} $d(e,M) = 0$, or, in other words, $e \in M$.

  \emph{Proof of claim.} Suppose that $d(e,M) > 0$. Since $x \mapsto d(x,M)$ is a continuous function, by compactness, there must be some finite set $M_0 \subset M$ of elements with finite distance to $a$ such that $F_0 \coloneqq \inf \{ f \in \bigcap \{ [a,c]^N: c \in M_0 \}\} > 0$. By \cref{prop:tree-proj}, there is a unique element $g \in \ccl(M_0) \subset M$ of minimal distance to $a$. It must be the case that $g \notin F_0$, so there must be some $m \in M_0$ such that $g \notin [m,a]$. Let $h$ be the unique element of $[m,a]$ of minimal distance to $g$.

  It is easy to check that in $\Ff(X)$, for any $A,B,C$, the unique element of $[A,B]$ closest to $C$ is contained in $[A,C]$, which implies that this is true for all models of $T(X)$. Therefore we have that $h \in [m,g]$. Since $[m,g]$ is in the algebraic closure of $mg$, we have that $[m,g]\subset M$. By \cref{lem:tree-contain-interval}, we have that $[m,g] \subseteq \ccl(M_0)$, and so $h \in \ccl(M_0)$, but $d(h,a) < d(g,a)$, which is a contradiction. \hfill $\square_{\text{claim}}$

  \emph{Claim.} $e$ is unique.

  \emph{Proof of claim.} Suppose that there are distinct $e$ and $e'$ in $F \cap M$. This implies that for any $m \in M$ with $d(m,a) < \infty$, $e$ and $e'$ are both in $[m,a]$. One of $e$ and $e'$ must be closer to $a$. Assume without loss of generality that $d(e,a) < d(e',a)$. Then we have that $e' \notin [e,a] \subseteq F$, which is a contradiction. \hfill $\square_{\text{claim}}$

  The argument for the last claim also establishes that for any $m \in M$, $d(m,a) \geq d(e,a)$.
\end{proof}

\begin{lem}\label{lem:big-distance}
  Fix $a,b,c,e \in M \models T(X)$ with pairwise finite distance. Let $c'$ be the unique closest point on $[a,b]$ to $c$ and $e'$ the unique closest point on $[a,b]$ to $e$. If $c' \neq e'$, then $d(c,e) = d(c,c') + d(c',e') + d(e',e)$.
\end{lem}
\begin{proof}
  It is easy to verify that this is true in $\Ff(X)$. For any $r, \e > 0$, the statement
  \begin{itemize}
  \item[] for any $a,b,c,e$ with pairwise distance $<r$, if $c'$ is the closest point on $[a,b]$ to $c$ and $e'$ to $e$ and $d(c',e') > \e$, then $d(c,e) = d(c,c') + d(c',e') + d(e',e)$ 
  \end{itemize}
  is first-order. Therefore these hold in any model of $T(X)$, which is precisely the desired conclusion.
\end{proof}

\begin{lem}\label{lem:portly-models}
  For any $\kappa$, there is a model $M$ of $T(X)$ with density character $\kappa$ such that $|M| = \kappa^\omega$.
\end{lem}
\begin{proof}
  This follows the argument in the proof of Theorem 8.10 in \cite{Carlisle2020}, noting that any given $\Rb$-tree embeds isometrically into a model of $T(X)$.
\end{proof}

\begin{prop}\label{prop:theory-is-stable}
  For any model $M\models T(X)$, the elements of $S_1(M)$ are precisely
  \begin{itemize}
  \item the realized types in $M$,
  \item types $p_{m,f}$ for each pair $m \in M$ and $f \in \Pf(X)$ with $f(0) = \tp_X(m)$, and
  \item types $q_x$ for each $x \in X$,
  \end{itemize}
  where
  \begin{itemize}
  \item $p_{m,f}$ is the type of an element $a \in N \succ M$ whose unique nearest element in $M$ is $m$ and which satisfies that $[m,a]$ corresponds to $f$ in $\Pf(X)$ and
  \item $q_x$ is the type of an element $b \in N \succ M$ with $\tp_X(b) = x$ and $d(b,M) = \infty$.
  \end{itemize}
  Furthermore, the metric\footnote{Recall that the `official' metric in $\Lc_X$ is $d_{\diam X}$.} on $S_1(M)$ is given by
  \begin{itemize}
  \item $d(p_{m,f},p_{m',f'}) = \min\{|f|+d(m,m')+|f'|,\diam X\}$ if $m\neq m'$,
  \item $d(p_{m,f},p_{m,f'}) = \min\{|f\sqcap f'|, \diam X\}$, where $f\sqcap f'$ is the longest common initial segment of $f$ and $f'$,
  \item $d(q_x,q_{x'}) = \der(x,x')$, and
  \item $d(m,q_x)=d(p_{m,f},q_x) = \diam X$
  \end{itemize}
  for any $m,m' \in M$, $x,x' \in X$, and $f,f' \in \Pf(X)$. So in particular,
  \[
    |M| + \dc X \leq \dc S_1(M) \leq |\Pf(X)|\cdot|M| + \dc X,
  \]
   and $T(X)$ is strictly stable.
\end{prop}
\begin{proof}
  Clearly every type $p(x)$ in $S_1(M)$ is either realized, satisfies $d(x,m) < \infty$ for some $m \in M$, or satisfies $d(x,m) = \infty$ for every $m \in M$. The characterizations of these types as $p_{m,f}$ and $q_x$ for various $m$,$f$, and $x$ follows from \cref{lem:project-onto-model}, \cref{prop:finite-type-characterization}, and \cref{cor:singleton-type-characterization}.

  For the metric on $S_1(M)$, the last three bullet points are clearly correct. The first bullet point is clearly an upper bound, so we just need to show that a smaller distance cannot be achieved. If $a,b \in N \succ M$ have nearest points $c,e \in M$, respectively, then these are also their nearest points on $[c,e] \subset M$, so by \cref{lem:big-distance}, $d(a,b) = d(a,c) + d(c,e) + d(e,b)$, as required.

  The bounds on the density character of $S_1(M)$ are obvious, so the fact that $T(X)$ is strictly stable follows from \cref{lem:portly-models}.
\end{proof}

\section{Main Theorem}

\begin{lem}\label{lem:fin-diam}
  Any compact topometric space $(X,\tau,\der)$ with open metric has finite diameter.
\end{lem}
\begin{proof}
  If some non-empty open subset $U \subseteq X$ has finite diameter, then by compactness there is some finite $\e > 0$ such that $X = U^{<\e}$, so $X$ has finite diameter.

  So assume that every non-empty open subset of $X$ has infinite diameter. Fix $x_1,y_1 \in X$ with $\der(x_1,y_1) > 1$. By lower semi-continuity, there are open neighborhoods $U_1 \ni x_1$ and $V_1 \ni y_1$ such that $U_1^{<1}$ and $V_1$ are disjoint.

  At stage $i$, given non-empty open sets $U_i$ and $V_i$, since $U_i$ and $V_i$ both have infinite diameter, we can find $x_{i+1} \in U_i$ and $y_{i+1} \in V_i$ with $\der(x_{i+1},y_{i+1}) > i+1$. We can find open neighborhoods $U_{i+1} \ni x_{i+1}$ and $V_{i+1}\ni y_{i+1}$ such that $\cl U_{i+1} \subseteq U_i$ and $\cl V_{i+1} \subseteq V_i$ and $U_{i+1}^{<i} \cap V_{i+1} = \varnothing$.

  Let $x_\omega$ be an element of $\bigcap_{i<\omega}\cl U_i$ and $y_\omega$ of $\bigcap_{i<\omega}\cl V_i$, which are both non-empty by compactness. By construction we have that $\der(x_\omega,y_\omega) > i$ for every $i < \omega$, but this contradicts that $\der$ is a metric (rather than an extended metric).
\end{proof}

\begin{thm}\label{thm:main-thm}
  For any compact topometric space $(X,\tau,\der)$ with an open metric, there is a stable continuous first-order theory $T$ such that $S_1(T)$ is isomorphic to $(X,\tau,\der)$.
\end{thm}
\begin{proof}
  If $X$ has a single point, then the theory of a one-point structure suffices, so assume that $X$ has more than one point. By \cref{lem:fin-diam}, $X$ has finite diameter, so we can form the theory $T(X)$. There is clearly a continuous $1$-Lipschitz map $f$ from $S_1(T)$ to $X$. We have by \cref{cor:singleton-type-characterization} that $f$ is a bijection, so it is a topological isomorphism. For any $a,b \in X$, there are, by construction, $K$ and $L$ in $\Ff(X)$ with $\tp_X(K) = a$ and $\tp_X(L)=b$ such that $d_{\diam X}(K,L) = d(K,L) = \der (a,b)$. Therefore $f$ is an isometry as well. Finally, by \cref{prop:theory-is-stable}, $T(X)$ is stable.
\end{proof}

It is natural to wonder if our main theorem can be improved by constructing a superstable theory $T$ with $S_1(T)$ isomorphic to a given $X$. In other words, are there any non-trivial restrictions on the topometry type of $S_1(T)$ for $T$ superstable? Clearly if $X$ is not CB-analyzable,\footnote{This is the topometric generalization of scatteredness. See \cite{BenYaacov2008}.} then any such $T$ cannot be $\omega$-stable or totally transcendental, but it is also possible that this is the only obstruction. 

\begin{quest}
  If $(X,\tau,\der)$ is a compact topometric space with an open metric, is there a superstable theory $T$ such that $S_1(T)$ is isomorphic to $X$?

  If $(X,\tau, \der)$ is CB-analyzable, is there a totally transcendental theory $T$ such that $S_1(T)$ is isomorphic to $X$?
\end{quest}

For comparison, note that every totally disconnected compact Hausdorff space is $S_1(T)$ for a superstable theory $T$, and every scattered compact Hausdorff space is $S_1(T)$ for a totally transcendental theory $T$.

There is also the task of characterizing higher type spaces. Even for $2$-types, there are new restrictions on what topometry types are possible. If $T$ has models with more than one element, then $S_2(T)$ has a non-trivial definable set, namely, $d(x,y) = 0$.


\bibliographystyle{plain}
\bibliography{../ref}

\end{document}